\documentclass[18pt,letterpaper]{article}
\usepackage[utf8]{inputenc}

\usepackage[english]{babel}

\usepackage{amsmath}
\usepackage{graphicx}
\usepackage[colorlinks=true, allcolors=blue]{hyperref}
\usepackage{amsfonts,amsmath,amssymb}
\usepackage{amsthm}
\usepackage{amstext}
\usepackage{graphicx}
\usepackage{mathtools}
\usepackage{enumitem}
\usepackage{float}
\usepackage{array}
\usepackage{latexsym}
\usepackage{hyperref,color,authblk,soul}
\usepackage{tikz}
\usepackage{caption}
\usepackage{subcaption}

\tikzstyle{vert}=[shape=circle,draw=black,fill=white, inner sep=.5mm]

\newtheorem{theorem}{Theorem}[section]
\newtheorem{definition}{Definition}[section]
\newtheorem{corollary}{Corollary}[section]

\newtheorem{remark}{Remark}[section]
\newtheorem{conjecture}[theorem]{Conjecture}

\begin{document}

\title{On the existence of $(r,g,\chi)$-cages
\thanks{Research supported by PAPIIT-M\'exico IN108121, Intercambio Acad\'emico, UNAM. Grant: "Buscando ciclos largos en gr\'aficas bipartitas" and CONACyT-M\'exico 282280.}}

\author{Gabriela Araujo-Pardo\\
Instituto de Matematicas. Universidad Nacional Aut\'onoma de M\'exico, M\'exico.\\
garaujo@im.unam.mx

\and
Zhanar Berikkyzy\\
Mathematics Department. Fairfield University \\
zberikkyzy@fairfield.edu

\and
Linda Lesniak\\
Department of Mathematics. Western Michigan University \\
linda.lesniak@wmich.edu}

\maketitle

%\linenumbers

\begin{abstract}
For integers $r \geq 2$ and $g \geq 3 $, an \emph{$(r,g)$-graph} is an $r$-regular graph with girth $g$, and an \emph{$(r,g)$-cage} is an $(r,g)$-graph of minimum order.  It is conjectured that all $(r,g)$-cages with even $g$ are bipartite, that is, have chromatic number 2.  Here we introduce the idea of an $(r,g,\chi)$-graph, an $r$-regular graph with girth $g$ and chromatic number $\chi$.  We investigate the existence of such graphs and study in detail the $(r,3,3)$-graphs of minimum order.  We also consider $(r,g,\chi)$-graphs for which there is a $\chi$-coloring where the color classes differ by at most 1.
\end{abstract}

\section{Introduction}

In this paper, we work with simple and finite graphs. We study a generalization of the \emph{Cage Problem}, which has been widely studied since cages were introduced by Tutte \cite{T47} in 1947 and after Erd\" os and Sachs \cite{ES63} proved their existence in 1963. An \emph{$(r,g)$-graph} is an $r$-regular graph in which the shortest cycle has length equal to $g$; that is, it is an $r$-regular graph with girth $g$. An \emph{$(r,g)$-cage} is an $(r,g)$-graph with the smallest possible number of vertices among all $(r,g)$-graphs; the order of an $(r,g)$-cage is denoted by $n(r,g)$. The Cage Problem consists of finding $(r,g)$-cages; it is well-known that $(r,g)$-cages have been determined only for very limited sets of parameter pairs $(r, g)$. There exists a simple lower bound for $n(r,g)$, given by Moore (see \cite{EJ13}), and denoted by $n_0(r,g)$. The cages that attain this bound are called \emph{Moore cages}.  
%This bound is obtained by counting the vertices of a tree, ${\cal{T}}_{(g-1)/2}$, rooted on a vertex and with radius $(g-1)/2$ if $g$ is odd; or the vertices of a ``double-tree'' rooted at an edge (that is, two different ${\cal{T}}_{(g-3)/2}$ trees rooted each one at the vertices incident with an edge) if $g$ is even.  These trees give the following bounds:

 % \begin{equation}\label{lowercages} n_0(r,g) = \left\{ \begin{array}{ll} 1 + r + r(r-1) + \cdots
%+ r(r-1)^{(g-3)/2} &\mbox{ if $g$ is odd};\\
%2(1 +(r-1) +
%\cdots + (r-1)^{g/2-1})&\mbox{ if $g$ is
%even}.\end{array}\right.\end{equation}

%\begin{definition}

%For integers $r \geq 2$ and $g \geq 3$, an $(r,g)-\textbf{graph}$ is an $r$-regular graph of fixed girth $g$.

%\end{definition}

%\begin{definition}

%An $(r,g)$-graph of minimum order $n(r,g)$ is called an $(r,g)-\textbf{cage}$.

%\end{definition}

The existence of cages for all such integers $r,g$ has been proved, but very few of these cages
have been determined.  However, based on these known cages, an interesting conjecture has been made (this conjecture appears in \cite{PBMO04}, but we can say that it is a \emph{ folklore Conjecture}, because it is very well-known among people interested on the Cage Problem).

\begin{conjecture}\label{cbipartite}
    
For even values of girth $g$, all $(r,g)$-cages are bipartite.

\end{conjecture}

This conjecture naturally leads to a relationship between $(r,g)$-cages and chromatic number since another way of stating this conjecture is the following:  

\begin{conjecture}\label{chromaticnumber2}(Conjecture \ref{cbipartite})
    
For even values of girth $g$, all $(r,g)$-cages have chromatic number equal to 2.

\end{conjecture}

Clearly, every $(r,g)$-cage with odd $g$ has chromatic number at least $3$. Furthermore, for any $g\geq 3$, any $(r,g)$-graph, independent of whether it is a cage or not, may have chromatic number at least $3$. 

\section{Definitions and background}
The ideas expressed previously motivate  the following definitions: 

\begin{definition}
An $(r,g,\chi)-\textbf{graph}$ is an $r\geq 2$-regular graph of fixed girth $g\geq 3$ and chromatic number equal to $\chi\geq 2$. 
\end{definition} 

\begin{definition}
An $(r,g,\chi)-\textbf{cage}$ is an $(r,g,\chi)$-graph of minimum order, and we denote by $n(r,g,\chi)$ the order of an $(r,g,\chi)$-cage.  
\end{definition}

The idea of this paper is the question of the existence of $(r,g,\chi)$-cages for any $r\geq 2$, $g\geq 3$ and $\chi\geq 2$. Note that if $\chi=k$ this implies that the graph is $k$-partite, but not $(k-1)$-partite so this problem is equivalent to asking what is the smallest $k$-partite $r$-regular graph with fixed girth $g$.  Therefore we can give  the following equivalent definitions: 

\begin{definition}
An $(r,g,k)-\textbf{graph}$ is a $k$-partite (non $(k-1)$-partite) $r$-regular graph, for $r\geq 2$ of fixed girth $g\geq 3$
\end{definition}

\begin{definition}
An $(r,g,k)-\textbf{cage}$ is an $(r,g,k)$-graph of minimum order, and we denote by $n(r,g,k)$ the order of an $(r,g,k)$-cage.  
\end{definition}

We will consider the existence of $(r,g,k)$-cages, or equivalently, the existence of $(r,g,\chi)$-cages.  In our following work we will use the terminology of $(r,g,\chi)$-cage.

By Brooks' Theorem, $(r,g,\chi)$-graphs exist only if $\chi \leq r+1.$  The authors of this paper do not know any result that proves the existence of $(r,g,\chi)$-graphs for $4\leq \chi \leq r+1$.   In Sections 3-5, we focus on the problem for $\chi=3$.  

In Section 3 we determine infinite classes of $(r,3,3)$-cages and we note that in all $3$-colorings of these cages the size of the color classes differ by at most one. This motivates the definition of $(r,g,\chi)$-equitable graph that we study in Section 5.

In Section 4, assuming Conjecture \ref{cbipartite}, we establish
the existence of $(r,g,3)-$graphs for any $r\geq 2$, $g\geq 3$. Notice that, also by the Conjecture \ref{cbipartite}, the $(r,g)$-cages for even $g\geq 4$, are $(r,g,2)$-cages. 

%It is natural to ask if an $(r,g)$-cages of odd girth $g\geq 3$ and chromatic number $\chi\geq 3$ is also an $(r,g,\chi)$-cage.

%ISN'T THIS ALWAYS TRUE?  SPSE G IS AN (R,G)-CAGE,  WHICH HAS CHROMATIC NUMBER CHI.  THEN IT IS AN (R,G,CHI)-GRAPH BY DEFINITION.  IF IT IS NOT A (R,G CHI)-CAGE, THEN THERE IS AN (R,G,CHI)-GRAPH H OF SMALLER ORDER.  BUT THEN H IS AN (R,G)-GRAPH OF SMALLER ORDER THAN G.  CONTRADICTION.

In particular, by Brooks' Theorem we know that every  $(3,g)$-cage of odd girth $g\geq 3$, except $K_4$, has chromatic number 3. Then the $(3,g)$-cages, except $K_4$ which is a $(3,3,4)$-cage, are exactly the $(3,g,3)$-cages.  In this paper we are particularly interested in $(r,g,\chi)$-graphs of odd girth $g$.  % and to continue with this aim, in Section \ref{$(r,3,3)$} we determine the $(r,3,3)$-cages and we note that in all $3$-colorings of these cages the size of the color classes differ at most in one. This motivates the definition of $(r,g,\chi)$-equitable graph. 

%This definition is motivated by the well-known Conjecture, for the people interested in the {\bf{ Cage Problem}}, which states that all the cages of even girth are bipartite (see \cite{PBMO04}). As all bipartite cages have a chromatic number equal to $2$ is naturally asking about the chromatic number of the cages of girth even, but before that, it is important asking about the existence of $(r,g)$-graphs with fixed odd girth and a chromatic number equal to three. 

\bigskip

\section{$(r,3,3)$-cages}\label{$(r,3,3)$}

In this section we focus on $(r,3,3)$-graphs and determine $n(r,3,3)$ for $r \geq 2.$  Moreover, we describe an infinite class of $(r,3,3)$-cages for all $r$  and,  based on the parity of $r$, essentially determine all such cages.

In the last case in the proof of Theorem \ref{r33}, we use the following result of Moon and Moser \cite{MM63}.

\begin{theorem}\label{MM63} If $G$ is a balanced bipartite graph with $2n$ vertices having minimum degree $\delta(G) > n/2$, then $G$ is hamiltonian.

\end{theorem}

\begin{theorem}\label{r33}
For $r \geq 2$ we have:

\[n(r,3,3) = \left\{ \begin{array}{cr}
     r  +  r/2  &\mbox{if $r$ is even;} \\
     r +  \lceil r/2 \rceil + 1 &\mbox{if $r$ is odd, and $\lceil r/2 \rceil$ is even;} \\
     r + \lceil r/2 \rceil + 2 &\mbox {if $r$ is odd, and $\lceil r/2 \rceil$ is odd.}
     \end{array} \right. \]
\end{theorem}

\begin{proof}
We first determine a lower bound on $n(r,3,3)$.  Let $G$ be an $r$-regular graph with $\chi(G) = 3.$  Color $G$ with 3 colors, resulting in color classes $A$, $B$ and $C,$ with $|A| \geq |B| \geq |C|.$  Since $G$ is $r$-regular, $|A| + |B| \geq r, |B| + |C| \geq r$ and $| A| + |C| \geq r.$  Consequently $2|A| + 2|B| + 2|C| \geq 3r,$ that is, $|V(G)| \geq \frac{3r}{2}$ and so $|V(G)| \geq r + \lceil r/2 \rceil. $

In the case $r$ is even, the $(r,3,3)$-graph $ K_{r/2,r/2,r/2} $  attains this bound and so $n(r,3,3) = r + r/2. $

In the case $r$ is odd and $\lceil r/2 \rceil $ is even, it follows that $r + \lceil r/2 \rceil$ is odd and so no $r$- regular graph of order $r + \lceil r/2 \rceil$ exists.  Thus $n(r,3,3) \geq r + \lceil r/2 \rceil + 1.$   The graph $G = K_{\lceil r/2 \rceil, \lceil r/2 \rceil, \lceil r/2 \rceil}$ has order $r + \lceil r/2 \rceil + 1$ which is even and is $(r + 1)$ -regular.  Removing any 1-factor from $G$ results in an $r$-regular graph with girth 3 and chromatic number 3.  Therefore $n(r,3,3) = r + \lceil r/2 \rceil + 1.$

Finally, if $r$ is odd and $\lceil r/2 \rceil $ is odd, it follows that $r + \lceil r/2 \rceil$ is even. Suppose that $|V(G)=r+\lceil r/2 \rceil$, then it follows that $|A| = |B| = \lceil r/2 \rceil$ and $|C| = \lfloor r/2 \rfloor.$ Notice that every vertex in $A$ is adjacent to every vertex in $B \cup C$ and every vertex in $B$ is adjacent to every vertex in $A \cup C.$  But then each vertex in $C$ has degree $r+1$, which contradicts the assumption that $G$ is $r$-regular.  Thus, $|V(G)\leq  r + \lceil r/2 \rceil + 1$, which is odd. Since $G$ is $r$-regular with odd $r$, $|V(G)|$ must be even and so $n(r,3,3)\geq r + \lceil r/2 \rceil + 2$.  We will construct an $(r,3,3)$ graph that attains this bound. Consider the graph $G = K_{(r+3)/2,(r+1)/2,(r+1)/2}$, it has order $r + \lceil r/2 \rceil + 2$, chromatic number 3, and girth 3. Define $A, B, C$ as above, so that every vertex of $A$ has degree $r + 1$ in $G$  and every vertex of $B \cup C$ has degree $r+2$ in $G.$ Select $x\in A$, $y\in B$ and $z \in C.$  We can remove any 1-factor from $G$, that contains none of the edges $xy,yz$ or $xz$ to create a graph $H$ with girth 3 and chromatic number 3.  Furthermore, every vertex of $A$ has degree $r$ in $H$ and every vertex of $B \cup C$ has degree $r + 1$ in $H$.  Moreover, the subgraph $F$ of $H$ induced by $B \cup C$ in $H$ is a spanning subgraph of $K_{(r+1)/2,(r+1)/2}$ in which every vertex has degree at least $(r-1)/2 > (r+1)/4,$ provided $r > 3.$  This is true since $\lceil r/2 \rceil$ is odd.  In this case,  $F$ is hamiltonian (see \cite{MM63}) and so $F$ has a 1-factor that does not contain the edge  $yz$.  We can remove this 1-factor to create an $r$-regular graph with girth 3 and chromatic number 3 of order $r + \lceil r/2 \rceil + 2$. Therefore, $n(r,3,3) = r + \lceil r/2 \rceil + 2.$

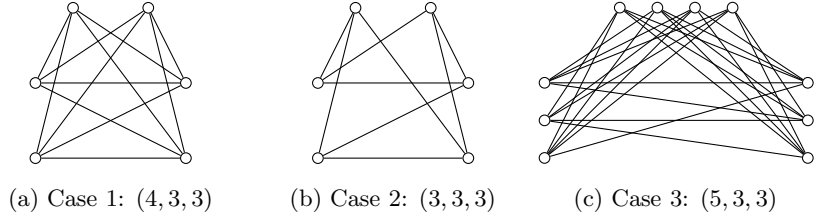
\begin{figure}
\centering
\centering
\begin{subfigure}[c]{.3\textwidth}
\centering
\begin{tikzpicture}  
\node[vert] (x1) at (0,0) {};
\node[vert] (x2) at (2,0) {};
\node[vert] (y1) at (0,1) {};
\node[vert] (y2) at (2,1) {};
\node[vert] (z1) at (.5,2) {};
\node[vert] (z2) at (1.5,2) {};

\foreach \x / \y in {x1/x2,x1/y2, x1/z2, x2/y1, x2/z1,y1/y2,y1/z1,y2/z2, y2/z1,y1/z2,x1/z1,x2/z2}
    \draw (\x) -- (\y);

\end{tikzpicture}\caption{Case 1: $(4,3,3)$}
\end{subfigure}
\begin{subfigure}[c]{.3\textwidth}
\centering
\begin{tikzpicture}  
\node[vert] (x1) at (0,0) {};
\node[vert] (x2) at (2,0) {};
\node[vert] (y1) at (0,1) {};
\node[vert] (y2) at (2,1) {};
\node[vert] (z1) at (.5,2) {};
\node[vert] (z2) at (1.5,2) {};

\foreach \x / \y in {x1/x2,x1/y2,  x2/z1,y1/y2,y1/z1,y2/z2,y1/z2,x1/z1,x2/z2}
    \draw (\x) -- (\y);

\end{tikzpicture}
\caption{Case 2: $(3,3,3)$}
\end{subfigure}
\begin{subfigure}[c]{.3\textwidth}
\centering
\begin{tikzpicture}  
\node[vert] (1) at (0,0) {};
\node[vert] (2) at (.5,0) {};
\node[vert] (3) at (1,0) {};
\node[vert] (4) at (1.5,0) {};

\node[vert] (5) at (-1,-1) {};
\node[vert] (6) at (-1,-1.5) {};
\node[vert] (7) at (-1,-2) {};

\node[vert] (8) at (2.5,-1) {};
\node[vert] (9) at (2.5,-1.5) {};
\node[vert] (10) at (2.5,-2) {};

\foreach \x / \y in {1/6,1/7,1/8,1/9,1/10,
					 2/5,2/7,2/8,2/9,2/10,
					 3/5,3/6,3/7,3/8,3/10,
					 4/5,4/6,4/7,4/9,4/10,
					 5/8,5/9,
					 6/9,6/10,
					 7/8}
    \draw (\x) -- (\y);

\end{tikzpicture}
\caption{Case 3: $(5,3,3)$}
\end{subfigure}
\caption{\label{fig:r33} The cases of $(r,3,3)$-cages}
\end{figure}

\end{proof}

Figure \ref{fig:r33} has examples of the three constructions in the proof of Theorem \ref{r33}. 

We note that this theorem essentially describes the $(r,3,3)$-cages in the first two cases.  Also, in each of the $3$-colorings described in the proof, the sizes of the color classes differ by at most 1.  In Section 5 such colorings will be called equitable colorings.

\section{Existence of $(r,g,3)$-graphs}\label{existence}

In this section, we will prove, assuming  Conjecture \ref{cbipartite}, that  $(r,g,3)$-graphs exist for all $r\geq 2$ and any girth $g\geq 3$. 

\begin{theorem}\label{existencerg3}
The $(r,g,3)$-graphs exist for any $r\geq 2$ and $g\geq 3$. 
\end{theorem}
\begin{proof}
We will divide the proof into two cases depending on the parity of the girth. 
\begin{itemize}
\item{\bf{Case 1:}} Suppose that girth $g\geq 3$ is odd. Clearly a $g$-cycle is a $(2,g,3)$-graph, in fact, it is a $(2,g,3)$-cage and Brooks' Theorem, together with the existence of $(3,g)$-cages, implies the existence of $(3,g,3)$-graphs, or $(3,g,3)$-cages, for odd $g\geq 3$. We now prove the existence of $(r,g,3)$-graphs for $r\geq 4$ and $g\geq 3$ odd. Let $G$ be an $(r,g+1)$-cage and,  as $g$ is odd, by Conjecture \ref{chromaticnumber2}, $G$ is a  $(r,g+1,2)$-cage.

Now we divide the proof in other two sub-cases, depending of the parity of $r$:

\begin{itemize}
    \item {\bf{Case (1a):}} Suppose that $r\geq 4$ is even.  Let $x$ be a vertex in a $(g+1)$- cycle and let $N(x)=\{x_1,x_2,\ldots, x_r\}$, ordering these neighbours with the property that, at least, $\{x_1,x,x_2\}$ is in a $(g+1)$-cycle.  Let $H=G-x$; we construct a new graph $H'$ with $V(H')=V(H)$ and $E(H')=E(H)\cup \{x_i,x_{i+1}\}$ for odd $1\leq i\leq r-1$. As $d_G(x_i,x_j)=2$, for any $1\leq \{i,j\}\leq r $ and $i\not=j$, and $\{x_1,x,x_2\}$ is in a $(g+1)$-cycle in $G$, the edge $\{x_1x_2\}$ is in a $g$-cycle in $H'$, and we obtain an $(r,g)$-graph. 
    As $g$ is odd, the chromatic number of $H'$ is at least $3$.  To finish, we will give a $3$-coloring of $H'$. 
    As $G$ is bipartite, if $f:V(G)\longrightarrow \{0,1\}$ is a $2$-coloring of $G$, it follows that if $f(x)=0$ then $f(N(x))=1$. We give a $3$-coloring of $H'$ as $f'(V(H'))\longrightarrow \{0,1,2\}$ with $f(x_j)=2$ for all even $2\leq j\leq r$ and $f'(z)=f(z)$ in any other case. 
    \item {\bf{Case (1b):}} Suppose that $r\geq 4$ is odd. Take an $xy$-edge, let $N(x)-  \{y\}=\{x_1,x_2,\ldots, x_{r-1}\}$ and $N(y)-\{x\}=\{y_1,y_2,\ldots, y_{r-1}\}$, with the property that also here $\{x_1,x,x_2\}$ is in a $(g+1)$-cycle. As $r$ is odd, $r-1$ is even; let $H=G-\{x,y\}$.  Then $H'$ a graph such that $V(H')=V(H)$ and $E(H')=E(H)\cup \{x_ix_{i+1}, y_iy_{i+1}\}$, for odd $1\leq i\leq r-2$. As in the previous case, the edge $\{x_1x_2\}$ is in a $g$-cycle in $H'$, and we obtain an $(r,g)$-graph. 
    Also, as before, the chromatic number of $H''$ is at least $3$ and we give a $3$-coloring $f'$ of $H'$ taking into account the $2$-coloring of $G$ as before.  Here $f'(V(H'))\longrightarrow \{0,1,2\}$ with $f(\{x_j,y_j\})=2$ for all even $2\leq j\leq r-1$ and $f'(z)=f(z)$ in any other case. 
    \end{itemize}
    In both cases, we obtain an $(r,g,3)$-graph for odd $g\geq 3$.
\item{\bf{Case 2:}} Suppose the girth $g \geq 4$ is even and
let $G$ be an $(r,g)$-cage. By Conjecture \ref{chromaticnumber2}, $G$ is an  $(r,g,2)$-cage. First, we will note that in $G$,  there is no edge $e$ that is in all the cycles of length $g$. To prove this statement we will use the constructions given in {\bf{Case 1}}. Suppose that the statement is false and that all the $g$-cycles in $G$ contain $xy=e$. 
Here, divide the proof into two cases as before, depending on the parity of $r$. As we will use the previous constructions, we will give the proof without many details. Suppose that $r$ is even. Delete the vertex $x$ and, as a consequence, we also delete the edge $e=xy$. By hypothesis, the new graph has girth at least $g+1$, all the vertices have degree $r$ except a set of them (all of degree $r-1$ and mutually at distance $2$). We make the same construction as in {\bf{Case 1a}} and obtain a $r$-regular graph of girth $g$ of order $n(r,g)-1$, which is a contradiction. In the other case, if $r$ is odd, we delete the edge $e=xy$ and repeat the construction given in {\bf{Case 1b}}.   Since we delete the edge $e$, none of the cycles in the new graph has length smaller than $g+1$. Also, in this construction, we only add edges between vertices that, in the original graph, were at distance two and that  were not incident with the edge $e$, that is all of the new cycles have length at least $g+1$. Then, we obtain an $r$-regular graph of girth $g$ and order $n(r,g)-2$, which also is a contradiction. 

Consequently, given an edge $xy$  in $G$ contained in a cycle of length $g$ there exists at least one other cycle of length $g$ that does not contain $xy$. Subdividing the edge $xy$ with a vertex $z$ of degree two, we obtain a graph $H$, such that $V(H)=V(G)\cup \{z\}$ and $E(H)=E(G)\cup \{xz,zy\}$, where $xyz$ is a $2$-path in $H$. As $xy$ was in a $g$-cycle in $G$, now the $xzy$-path is in a $(g+1)$-cycle of odd length (because $g$ is even) in $H$.  Therefore, the chromatic number of $H$ is at least $3$ and if $G$ is colored as before, the coloring of $H$ preserves all the colors of $G$ and adds a new color in $z$. Notice also that, as there exists a $g$-cycle that does not contain $xy$ the girth of the graph is still $g$. 
Here we also give two different constructions depending of the parity of $r$. 
\begin{itemize}
    \item {\bf{Case (2a):}} If $r$ is even, we take $\frac{r}{2}$ copies of $H$ and we identify all the vertices $z$ in each copy. All the vertices have degree $r$ and the graph has girth $g$ because all of these graphs have $g$-cycles that do not include the edge $xy$. 
    \item {\bf{Case (2b):}} If $r$ is odd, first of all take a copy of $H$, called $H'$. Clearly, the chromatic number of $H$ is equal to three, and $H$ inherited the coloring of $G$ except for the new vertex $z$ that has color $2$. Recoloring $H'$ exchanging the chromatic classes of colors $0$ and $2$ of $H$ in $H'$, that is, $H'$ is coloring with colors $1$ and $2$ in all the vertices, except one vertex, called $z'$, that receives the color $0$. Join these two graphs adding the edge $zz'$. Here all the vertices have degree $r$ except $z$ and $z'$ that have degree $3$. Take $\frac{r-3}{2}$ copies of $H$ and $H'$ (preserving their colorings) and identify them in $z$ and in $z'$ respectively. In this new graph called $H''$, all the vertices have degree $r$, the degree of the vertices different than $z$ and $z'$ is clearly $r$ because it does not change and the degree of $z$ (and $z'$) is now $r-3+3=r$. \end{itemize}

In both cases, we obtain $(r,g,3)$-graphs, and we prove their existence. 
\end{itemize}
\end{proof}

As a corollary of this theorem, we have that: 

\begin{corollary}\label{existencerg3cages}
The $(r,g,3)$-cages exist for any $r\geq 2$ and $g\geq 3$. 
\end{corollary}

\section{Equitable colorings}\label{equitable}

%
%but we can divide the proof in two cases, for odd girth and for even girth. 
%Notice newly that the existence of $(3,g,3)$-graphs, for odd $g\geq 3$ is given by Brook´s Theorem we will pro and    In particular, Brook´s Theorem guarantees the existence of $(3,g,3)$-graphs of odd girth $g\geq 3$. Moreover, we also prove the existence of the $(4,g,3)$-graphs with fixed odd $g\geq 3$. Then, we define: 

%Notice that, $G$ is a graph with chromatic number equal to $3$, if and only if, $G$ is a tripartite graph which is not bipartite.

%In Figure \ref{fig:tripartite} we show an example of a $(10,3,3)$-cage (notice that between any pair of "nails" with different colors we have an edge and the sets of nails with the same colors is an independent set. 
%{\color{red} Here we have to prove the lower and the upper bounds, and obviously, the picture is a "jog".

%\begin{figure}
%\centering
%\includegraphics[width=0.3\textwidth]{tripartite.jpg}
%\caption{\label{fig:tripartite} A $(10,3,3)$-equitable cage.}
%\end{figure}

%We are also

%In this section we are interested in the existence of these graphs that are called \textbf{equitable}, that is,  that have chromatic number 3 and there is a 3-coloring in which the orders of the color classes differ at most 1.  

\begin{definition} 

Let $G$ a graph and $\chi$ a vertex coloring of $G$. We say that $\chi$ is an \textbf{equitable coloring} of $G$ is all the chromatic classes differ at most by one.
\end{definition}

\begin{definition} 

An $(r,g,\chi)$- graph $G$ is called an $(r,g,\chi)$-\textbf{equitable graph} if there exits an $\chi$-equitable coloring of $G$. Moreover, $G$ is an $(r,g,\chi)$-\textbf{equitable cage} if $G$ is an $(r,g,\chi)$-equitable graph of minimum order. 
\end{definition}

\begin{remark}
All the $(r,3,3)$-cages constructed in the proof of Theorem \ref{r33} are $(r,3,3)$-equitable cages.
\end{remark}

%As shown in the proof of  Theorem \ref{r33}, we see that $K_{r/2,r/2, r/2}$ is %the only $(r,3,3)$-cage for  even $r \geq 4,$ and that this graph is actually an %equitable-$(r,3,3)$-cage. Furthermore,  in the proof of Theorem \ref{r33}, we 

%all $(r,3,3)$-cages for odd $r$, and clearly all of these cages are equitable.

We next show, assuming Theorem 4.1, that $(r,g,3)$-equitable graphs exist for all $r \geq 2$ and $g \geq 3.$

\begin{theorem}

The $(r,g,3)$-equitable graphs exist for all $r \geq 2$ and $g \geq 3$.

\end{theorem}

\begin{proof}

Let $G$ be an $(r,g,3)$-graph and consider a 3-coloring of $G$ with color classes $C_1, C_2, C_3.$  Assume further that $|C_1| = x, |C_2|=y , |C_3| = z.$

Construct a new graph $G^*$ consisting of three disjoint copies $G_1, G_2, G_3$ of $G.$  Then $G^*$ is $r$-regular, $g(G^*)=r,$ and $\chi(G^*) = 3$.  Let $C_{ij}$ denote the copy of $C_i$ in $G_j.$  We give an equitable 3-coloring of $G^*$ as follows.

Define $C_1^* = C_{11} \cup  C_{22} \cup C_{33}, C_2^* = C_{21} \cup C_{32} \cup C_{13} , C_3^* = C_{31} \cup C_{12} \cup C_{23}.$ Then $|C_i^*| = x+y+z$ for $i=1,2,3.$  Thus $C_1^*,C_2^*,C_3^*$ is an equitable 3-coloring of $G^*$, and so $G^*$ is an equitable $(r,g,3)$-graph.

\end{proof}

As a corollary of this theorem, we also have that: 

\begin{corollary}\label{existenceequitable}
The $(r,g,3)$-equitable cages exist for all $r \geq 2$ and $g \geq 3.$ 
\end{corollary}

%\subsection{Odd girth $g\geq 5$ and $r=3$}

For all odd $g$ every $(3,g)$-cage is a $(3,g,3)$-cage, except $K_4$. The question here is if they have equitable 3-colorings when $g \geq 5$. 

We conjecture that the answer is ``yes" and this is motivated by the fact that many of the known cages, including the  Petersen Graph or the $(3,5)$-cage, the Mc Gee Graph or the $(3,7)$-cage, at least one of the eighteen $(3,9)$-cages (see \cite{BBS95}) and the $(3,11)$-cage (see Figure \ref{fig:(3,9)and(3,11)}) have equitable 3-colorings.

Recall that, by Conjecture \ref{chromaticnumber2} all $(r,g)$-cages of even $g$ are $(r,g,2)$-cages; for this reason, in the following, we will continue working only with cages of odd girth and we leave the study of  $(r,g,\chi)$ with $g$ even as an open problem.

To continue, in the following section we will analyze the question about the existence of $(r,g,\chi)$-cages with odd girth and $r\geq 4$. Moreover, we will prove that for $r=4$, $g=5$ and $\chi=3$, the unique $(4,5)$-cage Robertson graph, see Figure \ref{nonbalancedrobertson},  is the $(4,5,3)$-cage of order 19, but it is not a $(4,5,3)$-equitable cage. We will also exhibit a $(4,5,3)$-equitable cage of order $20$, showing that the question about the minimum order of an $(r,g,\chi)$-equitable graph (or the order of an $(r,g,\chi)$-equitable cage) make sense in general. 

\bigskip
%%%
%OK, LADIES.  PICKUP HERE.

%BUT WE NOW RUN INTO BOUNDED ORDER AND ITS MEANING

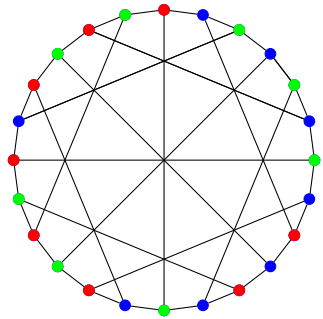
\begin{figure}
\centering
\begin{tikzpicture}[scale=2]  
\foreach \i in {1,...,24}
	{\node[vert,blue] (\i) at ({1*cos(deg(pi*\i /12))},{1*sin(deg(pi*\i/12))}) {};}

\foreach \i in {6,8,10,12,14,16,20,22}
	{\node[vert,red] (\i) at ({1*cos(deg(pi*\i /12))},{1*sin(deg(pi*\i/12))}) {};}
	
\foreach \i in {2,4,7,9,13,15,18,24}
	{\node[vert,green] (\i) at ({1*cos(deg(pi*\i /12))},{1*sin(deg(pi*\i/12))}) {};}

\foreach \x / \y in {1/2,2/3,2/3,3/4,4/5,5/6,6/7,7/8,8/9,9/10,10/11,11/12,12/13,13/14,14/15,15/16,16/17,17/18,18/19,19/20,20/21,21/22,22/23,23/24,1/24,
1/8, 2/19, 3/15,4/11,5/22,6/18,7/14,8/1,9/21,10/17,11/4,12/24,13/20,16/23}
    {
    \draw (\x) -- (\y);
    }
        
\end{tikzpicture}
\caption{\label{fig:balanced} The McGee graph is a $(3,7)$-cage, but it is also a $(3,7,3)$-equitable cage}
\end{figure}

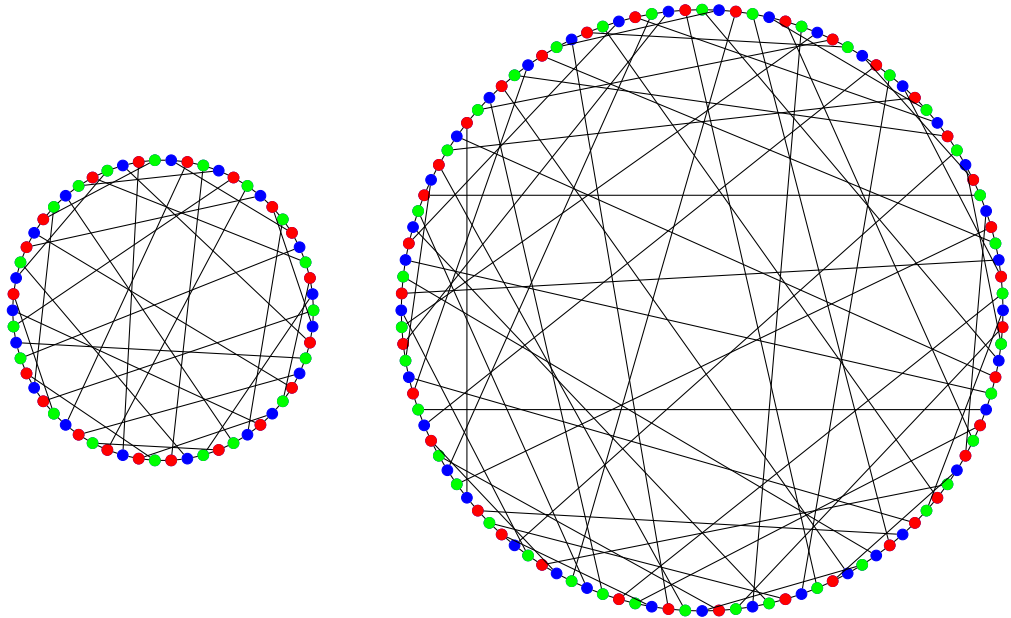
\begin{figure}
\centering
\begin{subfigure}[c]{.49\textwidth}
\centering
\begin{tikzpicture}[scale=2]  
\foreach \i in {1,...,58}
	{\node[vert,blue] (\i) at ({1*cos(deg(2*pi*\i /58))},{1*sin(deg(2*pi*\i/58))}) {};}

\foreach \i in {2,5,7,10,13,16,19,23,25,28,33,35,38,40,42,44,47,50,53,56}
	{\node[vert,red] (\i) at ({1*cos(deg(2*pi*\i /58))},{1*sin(deg(2*pi*\i/58))}) {};}
	
\foreach \i in {3,6,9,12,15,18,20,22,26,30,32,36,39,43,46,48,52,55,58}
	{\node[vert,green] (\i) at ({1*cos(deg(2*pi*\i /58))},{1*sin(deg(2*pi*\i/58))}) {};}

\foreach[remember=\x as \y (initially 1)] \x  in {2,...,58}
   	{
    \draw (\y) -- (\x);
    }
\draw (1) -- (58);

\foreach \x / \y in {1/46,2/52,3/19,4/32,5/14,6/49,7/57,8/25,9/40,10/30,11/20,12/44,13/37,15/23,16/41,17/56,18/27,21/48,22/34,24/53,26/45,28/36,29/50,31/55,33/43,35/58,38/54,39/47,42/51}
    {
    \draw (\x) -- (\y);
    }
        
\end{tikzpicture}
\end{subfigure}
\begin{subfigure}[c]{.49\textwidth}
\centering
\begin{tikzpicture}[scale=4]  
\foreach \i in {1,...,112}
	{\node[vert,blue] (\i) at ({1*cos(deg(2*pi*\i /112))},{1*sin(deg(2*pi*\i/112))}) {};}

\foreach \i in {2,5,8,11,14,17,20,23,26,29,32,35,38,41,44,47,49,52,55,58,61,64,69,71,74,79,82,85,89,92,96,98,100,103,105,108,111}
	{\node[vert,red] (\i) at ({1*cos(deg(2*pi*\i /112))},{1*sin(deg(2*pi*\i/112))}) {};}
	
\foreach \i in {1,4,7,10,13,16,19,22,25,28,31,34,37,40,43,46,50,54,57,59,62,65,67,70,73,76,78,80,83,86,88,91,94,99,101,104,107,110}
	{\node[vert,green] (\i) at ({1*cos(deg(2*pi*\i /112))},{1*sin(deg(2*pi*\i/112))}) {};}

\foreach[remember=\x as \y (initially 1)] \x  in {2,...,112}
   	{
    \draw (\y) -- (\x);
    }
\draw (1) -- (112);

\foreach \x / \y in {1/79,2/15,3/55,4/38,5/67,6/103,7/49,8/18,9/110,10/72,11/40,12/32,13/24,14/46,16/90,17/63,19/35,20/43,21/57,22/87,23/104,25/96,26/76,27/37,29/91,30/58,31/66,33/52,34/100,36/82,39/61,41/93,42/78,44/68,45/108,47/83,48/59,50/77,51/88,53/107,54/95,56/73,60/98,62/106,64/75,65/85,69/97,70/89,71/81,74/101,80/105,84/94,86/111,92/102,99/112,28/109}
    {
    \draw (\x) -- (\y);
    }
        
\end{tikzpicture}
\end{subfigure}
\caption{\label{fig:(3,9)and(3,11)} One of the eighteen $(3,9)$-cages and the unique $(3,11)$-cage are also a $(3,9,3)$- and $(3,11,3)$-equitable cages}
\end{figure}

\section{$(r,g,\chi)$-graphs with odd girth and $r\geq 4$
}

%{\color{blue} Another nice problem that we want to consider and it is related to the following group of graphs is the chromatic number of $(r;g)$-graphs, Moreover Zhanar and Gabriela conjecture that:

%Given any $(r,g,\chi)$-cage $G$ (or any $(r,g)$-cage $G$ with bounded order) and $r\geq 4$, we have that $\chi(G)\leq r-1$.

%The following cage of odd order that we have to analyze with this aim is the Roberson Graph that is the $(4,5)$-cage of order $19$, and also it is a $(4,5,3)$-cage. but we will proof Theorem \ref{nonbalancedrobertson}  it is non a $(4,5,3)$-equitable cage.

As indicated in the previous section, here we will show that the Robertson Graph, which is a $(4,5,3)$-cage, is not a $(4,5,3)$-equitable cage, that is, it has no equitable 3-coloring.

\begin{theorem}\label{nonbalancedrobertson}
The Robertson Graph is not a $(4,5,3)$-equitable cage. 
\end{theorem}
\begin{proof}

\begin{figure}
\centering
\begin{tikzpicture}[scale=2]  
\foreach \i in {1,...,19}
	{\node[vert,blue] (\i) at ({1*cos(deg(2*pi*\i /19))},{1*sin(deg(2*pi*\i/19))}) {};}

\foreach \i in {2,4,6,8,10,16,18}
	{\node[vert,red] (\i) at ({1*cos(deg(2*pi*\i /19))},{1*sin(deg(2*pi*\i/19))}) {};}
	
\foreach \i in {1,7,12,14,17}
	{\node[vert,green] (\i) at ({1*cos(deg(2*pi*\i /19))},{1*sin(deg(2*pi*\i/19))}) {};}

\foreach[remember=\x as \y (initially 1)] \x  in {2,...,19}
   	{
    \draw (\y) -- (\x);
    }
\draw (1) -- (19);

\foreach \x / \y in {1/5,1/16,2/13,2/9,3/18,3/7,4/15,4/12,6/17,6/13,7/11,8/19,8/15,9/17,10/5,10/14,14/18,16/11,17/9,17/6, 18/14,18/3,19/12}
    {
    \draw (\x) -- (\y);
    }
        
\end{tikzpicture}\label{robertsoncoloring}
\caption{A 3-non equitable coloring for Robertson cage}
\end{figure}
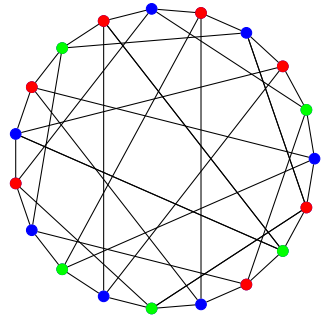

Let $G$ be the Robertson Graph.  In Figure \ref{nonbalancedrobertson} we see that there exists a subgraph of the $G$ that consists of four $5$-cycles and two special vertices marked in the picture. To make it easier to discuss, we will label these $5$-cycles as $\mathcal{C}_5^1=(x_1,\ldots,x_5)$, $\mathcal{C}_5^2=(x_5,\ldots,x_9)$, $\mathcal{C}_5^3=(x_9,\ldots,x_{13})$, $\mathcal{C}_5^4=(x_{13},\ldots,x_{17})$, two special vertices $X=\{x_{18}, x_{19}\}$, and three vertices that are common in two cycles $Y=\{x_5,x_9,x_{13}\}$. 

It is known that $\chi(G)=3$. We also know that there exists a $3$-coloring of Robertson graph with chromatic classes of sizes $\{5,7,7\}$ (see Figure \ref{nonbalancedrobertson}). Notice also that this $3$-coloring is not equitable. 
 If there exists an equitable coloring for the  Robertson Graph it must have chromatic classes of sizes $\{6,6,7\}$.

We note that in any $3$-coloring of the Robertson Graph,  there can be at most two vertices of one chromatic class in each pentagon defined before and at most one element of one chromatic class also in $X$. 

%Notice also that this $3$-coloring is not equitable. If there exists an equitable coloring for Robertson Graph it should have chromatic classes of sizes $\{6,6,7\}$. 
\begin{figure}
\centering
\begin{tikzpicture}[scale=1]  
\usetikzlibrary{calc}
\def\r{2}

\foreach \i in {1,...,19}
	{\node[vert,blue,label=above:{}] (\i) at ({\r*cos(deg(2*pi*\i /19))},{\r*sin(deg(2*pi*\i/19))}) {};}	

\foreach \i in {1,...,19}
	{\node[label=deg(-2*pi*\i/19+2*pi*8 /19):{$x_{\i}$}] (x\i) at ({\r*cos(deg(-2*pi*\i /19+2*pi*8 /19))},{\r*sin(deg(-2*pi*\i/19+2*pi*8 /19))}) {};}

\foreach[remember=\x as \y (initially 1)] \x  in {2,...,19}
   	{
    \draw (\y) -- (\x);
    }
\draw (1) -- (19);

\foreach \x / \y in {1/5,1/16,2/13,2/9,3/18,3/7,4/15,4/12,6/17,6/13,7/11,8/19,8/15,9/17,10/5,10/14,14/18,16/11,17/9,17/6, 18/14,18/3,19/12}
    {
    \draw (\x) -- (\y);
    }

\begin{scope}[rotate=0]
\draw[ultra thick, domain=-3*pi:3*pi,samples=500] plot[domain=-3*pi/19:pi/19+3*2*pi/19] ({deg(\x)}:{2.3}) -- plot[domain=pi/19+3*2*pi/19:-3*pi/19] ({deg(\x)}:{1.7}) -- cycle;
\end{scope}

\begin{scope}[rotate=deg(4*2*pi/19)]
\draw[ultra thick, domain=-3*pi:3*pi,samples=500] plot[domain=-3*pi/19:pi/19+3*2*pi/19] ({deg(\x)}:{2.2}) -- plot[domain=pi/19+3*2*pi/19:-3*pi/19] ({deg(\x)}:{1.8}) -- cycle;
\end{scope}

\begin{scope}[rotate=deg(-4*2*pi/19)]
\draw[ultra thick, domain=-3*pi:3*pi,samples=500] plot[domain=-3*pi/19:pi/19+3*2*pi/19] ({deg(\x)}:{2.2}) -- plot[domain=pi/19+3*2*pi/19:-3*pi/19] ({deg(\x)}:{1.8}) -- cycle;
\end{scope}

\begin{scope}[rotate=deg(-8*2*pi/19)]
\draw[ultra thick, domain=-3*pi:3*pi,samples=500] plot[domain=-3*pi/19:pi/19+3*2*pi/19] ({deg(\x)}:{2.3}) -- plot[domain=pi/19+3*2*pi/19:-3*pi/19] ({deg(\x)}:{1.7}) -- cycle;
\end{scope}

\begin{scope}[rotate=deg(-11.25*2*pi/19)]
\draw[ultra thick, domain=-3*pi:3*pi,samples=500] plot[domain=0:3*pi/19] ({deg(\x)}:{2.2}) -- plot[domain=3*pi/19:0] ({deg(\x)}:{1.8}) -- cycle;
\end{scope}

\end{tikzpicture}
\caption {\label{robertsonlabel} The $(4,5)$-Robertson cage}
\end{figure}
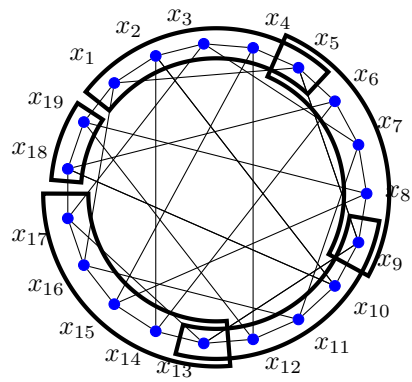

 %We note that, by properties of $3$-colorings, we can select at most two vertices of one chromatic class in each pentagon defined before and at most one element of one chromatic class also in $B$.
 
Suppose that there exists an equitable $3$-coloring of Robertson Graph, and the chromatic classes of this coloring are given by the sets $\mathcal{R}$, $\mathcal{B}$, $\mathcal{G}$ corresponding to colors $red$, $blue$, $green$, and suppose wlog $|\mathcal{R}|\geq |\mathcal{B}| \geq |\mathcal{G}|$. Then we have that  $|\mathcal{R}|= 7$, and $|\mathcal{B}|=|\mathcal{G}|= 6$.

To start the proof, we will analyze the structure of  $\mathcal{R}$, the chromatic class of order seven.  Recall that there can be at most two vertices of $\mathcal{R}$ in any of the four given pentagons and at most one in $Y$. 
 
We will give some observations that we will prove: 

\begin{itemize}
\item It is not possible that there are two red vertices in two different sets given by $\mathcal{C}_5^i \setminus Y$, for $i\in \{1,2,3,4\}$. 

Notice that if there exist two red vertices in $\mathcal{C}_5^2 \setminus Y$ they must be $\{x_6,x_8\}$ and immediately no element of $X=\{x_{18},x_{19}\}$ can be red. And similarly, if we have two red vertices in $\mathcal{C}_5^3 \setminus Y$ they are  $\{x_{10},x_{12}\}$.

Suppose that we have two red vertices in $\mathcal{C}_5^1 \setminus Y$ and two red vertices in $\mathcal{C}_5^2 \setminus Y$.  In the latter, as mentioned, the red vertices are $\{x_6,x_8\}$. The red vertices in $\mathcal{C}_5^1 \setminus Y$  are either $\{x_1,x_3\}$ or $\{x_2,x_4\}$ .

First, suppose that $\{x_1,x_3\} $ are the red vertices in $ \mathcal{C}_5^1 \setminus Y$. Then we have that: 

$$N(x_1,x_3,x_6,x_8)=\{x_2,x_4,x_5,x_7,x_9,x_{14},x_{15},x_{16},x_{17},x_{18},x_{19}\}.$$

Since the red chromatic class has order $7$ we must have four more red vertices, but they can be neither in $\mathcal{C}_5^4 \setminus Y$ nor in $X$.  So all of them must be in $\mathcal{C}_5^3$ and this is not possible.

Now, suppose that  $\{x_2,x_4\}\in \mathcal{C}_5^1 \setminus Y$ are red.  We have that: 
$$N(x_2,x_4,x_6,x_8)=\{x_1,x_3,x_5,x_7,x_9,x_{10},x_{12},x_{14},x_{15},x_{16},x_{18},x_{19}\}.$$

Here, only  $x_{11},x_{13}$ in $\mathcal{C}_5^3$ can be red. But $x_{11}x_ {17}\in E(G)$ and this implies that no vertex of $\mathcal{C}_5^4$ other than  $x_{13}$ can be red and no  vertices in $X$ are red.  But then there are fewer than 7 red vertices, a contradiction.

A similar analysis must be done if we have two red vertices in $(\mathcal{C}_5^1 \cup \mathcal{C}_5^3) \setminus Y$, $(\mathcal{C}_5^2 \cup \mathcal{C}_5^4) \setminus Y$ or $(\mathcal{C}_5^3 \cup \mathcal{C}_5^4) \setminus Y$. 

Then we only have two cases, where we have two red vertices in $\mathcal{C}_5^1 \setminus Y$ or two red vertices in  $\mathcal{C}_5^4 \setminus Y$. First, if we assume two red vertices in $\mathcal{C}_5^4\setminus Y$ then we have two red vertices in $\mathcal{C}_5^2$ and one of them must be $x_{5}$.  The vertex $x_9$ cannot be red, and we must have one other red vertex in $\mathcal{C}_5^2$ and also only one more in $\mathcal{C}_5^3$ and one in $X$ (see the left side on Figure \ref{robertsonclase7}). For the other case, if we assume two red vertices in $\mathcal{C}_5^1 \setminus Y$, but only one vertex in $\mathcal{C}_5^4$ and it is different than $x_{13}$, then there must be three red vertices in $(\mathcal{C}_5^2 \cup \mathcal{C}_5^3)$ and one of them is $x_9$.  Finally, there is one red vertex in $X$ (see figure right on Figure \ref{robertsonclase7}). 

In the sequel, we will show that the colors of all the vertices in both cases are completely determined. Moreover, we will divide these two cases so that in the first $x_{5}\in \mathcal{R}$ and the rest is exactly the first case above and the second is $x_{9}\in \mathcal{R}$ and the rest is also a consequence (see Figure \ref{robertsonclase7} to understand the colorings). 

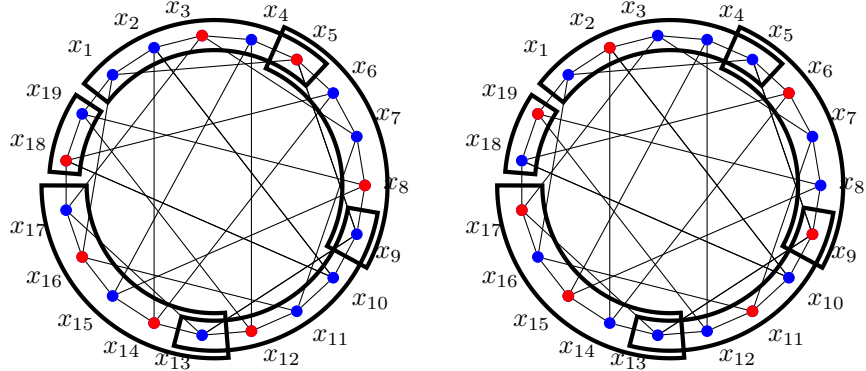
\begin{figure}
\centering
\begin{subfigure}{.49\textwidth}
\centering
\begin{tikzpicture}[scale=1]  
\usetikzlibrary{calc}
\def\r{2}

\foreach \i in {1,...,19}
	{\node[vert,blue,label=above:{}] (\i) at ({\r*cos(deg(2*pi*\i /19))},{\r*sin(deg(2*pi*\i/19))}) {};}	

\foreach \i in {1,...,19}
	{\node[label=deg(-2*pi*\i/19+2*pi*8 /19):{$x_{\i}$}] (x\i) at ({\r*cos(deg(-2*pi*\i /19+2*pi*8 /19))},{\r*sin(deg(-2*pi*\i/19+2*pi*8 /19))}) {};}

\foreach \i in {3,5,9,11,13,15,19}
	{\node[vert,red] (y\i) at ({\r*cos(deg(2*pi*\i /19))},{\r*sin(deg(2*pi*\i/19))}) {};}

\foreach[remember=\x as \y (initially 1)] \x  in {2,...,19}
   	{
    \draw (\y) -- (\x);
    }
\draw (1) -- (19);

\foreach \x / \y in {1/5,1/16,2/13,2/9,3/18,3/7,4/15,4/12,6/17,6/13,7/11,8/19,8/15,9/17,10/5,10/14,14/18,16/11,17/9,17/6, 18/14,18/3,19/12}
    {
    \draw (\x) -- (\y);
    }

\begin{scope}[rotate=0]
\draw[ultra thick, domain=-3*pi:3*pi,samples=500] plot[domain=-3*pi/19:pi/19+3*2*pi/19] ({deg(\x)}:{2.3}) -- plot[domain=pi/19+3*2*pi/19:-3*pi/19] ({deg(\x)}:{1.7}) -- cycle;
\end{scope}

\begin{scope}[rotate=deg(4*2*pi/19)]
\draw[ultra thick, domain=-3*pi:3*pi,samples=500] plot[domain=-3*pi/19:pi/19+3*2*pi/19] ({deg(\x)}:{2.2}) -- plot[domain=pi/19+3*2*pi/19:-3*pi/19] ({deg(\x)}:{1.8}) -- cycle;
\end{scope}

\begin{scope}[rotate=deg(-4*2*pi/19)]
\draw[ultra thick, domain=-3*pi:3*pi,samples=500] plot[domain=-3*pi/19:pi/19+3*2*pi/19] ({deg(\x)}:{2.2}) -- plot[domain=pi/19+3*2*pi/19:-3*pi/19] ({deg(\x)}:{1.8}) -- cycle;
\end{scope}

\begin{scope}[rotate=deg(-8*2*pi/19)]
\draw[ultra thick, domain=-3*pi:3*pi,samples=500] plot[domain=-3*pi/19:pi/19+3*2*pi/19] ({deg(\x)}:{2.3}) -- plot[domain=pi/19+3*2*pi/19:-3*pi/19] ({deg(\x)}:{1.7}) -- cycle;
\end{scope}

\begin{scope}[rotate=deg(-11.25*2*pi/19)]
\draw[ultra thick, domain=-3*pi:3*pi,samples=500] plot[domain=0:3*pi/19] ({deg(\x)}:{2.2}) -- plot[domain=3*pi/19:0] ({deg(\x)}:{1.8}) -- cycle;
\end{scope}

\end{tikzpicture}
\end{subfigure}
\begin{subfigure}{.49\textwidth}
\centering
\begin{tikzpicture}[scale=1]  
\usetikzlibrary{calc}
\def\r{2}

\foreach \i in {1,...,19}
	{\node[vert,blue,label=above:{}] (\i) at ({\r*cos(deg(2*pi*\i /19))},{\r*sin(deg(2*pi*\i/19))}) {};}	

\foreach \i in {1,...,19}
	{\node[label=deg(-2*pi*\i/19+2*pi*8 /19):{$x_{\i}$}] (x\i) at ({\r*cos(deg(-2*pi*\i /19+2*pi*8 /19))},{\r*sin(deg(-2*pi*\i/19+2*pi*8 /19))}) {};}

\foreach \i in {2,6,8,10,12,16,18}
	{\node[vert,red] (y\i) at ({\r*cos(deg(2*pi*\i /19))},{\r*sin(deg(2*pi*\i/19))}) {};}

\foreach[remember=\x as \y (initially 1)] \x  in {2,...,19}
   	{
    \draw (\y) -- (\x);
    }
\draw (1) -- (19);

\foreach \x / \y in {1/5,1/16,2/13,2/9,3/18,3/7,4/15,4/12,6/17,6/13,7/11,8/19,8/15,9/17,10/5,10/14,14/18,16/11,17/9,17/6, 18/14,18/3,19/12}
    {
    \draw (\x) -- (\y);
    }

%\draw[rounded corners={10pt}, very thick] (.68,1.56) foreach \x/\y in {1.56/.61,1.58/-.74,2.26/-.92,2.38/.8,1.09/1.99} {-- (\x,\y)} 
%            -- cycle;
%
%\draw[rounded corners={10pt}, very thick] (-.15,-2.62) foreach \x/\y in {1.37/-2.12, 2.4/-.88, 1.75/-.24, .9/-1.27, -.42/-1.7} {-- (\x,\y)} 
%            -- cycle;
%    
%\draw[rounded corners={10pt}, very thick] (-1.7,0) foreach \x/\y in {-1.38/1.05, -1.88/1.38,-2.42/.19} {-- (\x,\y)} 
%            -- cycle;    
%            
%\draw[rounded corners={10pt}, very thick] (-.46,-2.63) foreach \x/\y in {-2.38/-.51,-1.78/0.1,.41/-2.09} {-- (\x,\y)} 
%            -- cycle;  
%            
%\draw[rounded corners={10pt}, very thick] (-1.5,1.09) foreach \x/\y in {1.31/1.33, 1.3/2.19, -1.48/1.97} {-- (\x,\y)} 
%            -- cycle;                                          
%        

\begin{scope}[rotate=0]
\draw[ultra thick, domain=-3*pi:3*pi,samples=500] plot[domain=-3*pi/19:pi/19+3*2*pi/19] ({deg(\x)}:{2.3}) -- plot[domain=pi/19+3*2*pi/19:-3*pi/19] ({deg(\x)}:{1.7}) -- cycle;
\end{scope}

\begin{scope}[rotate=deg(4*2*pi/19)]
\draw[ultra thick, domain=-3*pi:3*pi,samples=500] plot[domain=-3*pi/19:pi/19+3*2*pi/19] ({deg(\x)}:{2.2}) -- plot[domain=pi/19+3*2*pi/19:-3*pi/19] ({deg(\x)}:{1.8}) -- cycle;
\end{scope}

\begin{scope}[rotate=deg(-4*2*pi/19)]
\draw[ultra thick, domain=-3*pi:3*pi,samples=500] plot[domain=-3*pi/19:pi/19+3*2*pi/19] ({deg(\x)}:{2.2}) -- plot[domain=pi/19+3*2*pi/19:-3*pi/19] ({deg(\x)}:{1.8}) -- cycle;
\end{scope}

\begin{scope}[rotate=deg(-8*2*pi/19)]
\draw[ultra thick, domain=-3*pi:3*pi,samples=500] plot[domain=-3*pi/19:pi/19+3*2*pi/19] ({deg(\x)}:{2.3}) -- plot[domain=pi/19+3*2*pi/19:-3*pi/19] ({deg(\x)}:{1.7}) -- cycle;
\end{scope}

\begin{scope}[rotate=deg(-11.25*2*pi/19)]
\draw[ultra thick, domain=-3*pi:3*pi,samples=500] plot[domain=0:3*pi/19] ({deg(\x)}:{2.2}) -- plot[domain=3*pi/19:0] ({deg(\x)}:{1.8}) -- cycle;
\end{scope}      
\end{tikzpicture}
\end{subfigure}
\caption {\label{robertsonclase7} Two different cases of a red chromatic class of size 7 in Robertson cage}
\end{figure}

\item In the following we will determine the coloring divided into two cases, when $x_5$ is red and when $x_9$ is red. Moreover in both of the cases, we will prove that it is not possible that the two chromatic classes $\mathcal{B}$ and $\mathcal{G}$  both have size $6$. In fact in both cases one of them has size $5$ and the other $7$. We conclude that it is not possible to have an equitable coloring of the Robertson Graph. 

\begin{itemize}

\item Suppose that $x_5$ is red and that $x_3$ is also red. Since $N(x_3,x_5)=\{x_2,x_4,x_6,x_7,x_9,x_{17}\}$, then the other red vertex in $\mathcal{C}_5^2$ must be $x_8$, and $$N(x_3,x_5,x_8)=\{x_2,x_4,x_6,x_7,x_9,x_{15},x_{17},x_{19}\}.$$
Since in $ \mathcal{C}_5^2\setminus Y$ there must be  two red vertices, the unique options are $x_{14}$ and $x_{16}$, and in $X$ the vertex $x_{18}$ must be red. 
In this case, we have that: 
 $$N(x_3,x_5,x_8,x_{14},x_{16},x_{18})=\{x_1,x_2,x_4,x_6,x_7,x_9,x_{10},x_{11},x_{13},x_{15},x_{17},x_{19}\}$$.
 
 Finally,  we have one red vertex in $ \mathcal{C}_5^3\setminus Y$ and it must be $x_{12}$, and we finish the description of the red vertices in the coloring (see Figure \ref{robertsonclase7}, left side).

 Now suppose that $x_5$ is red, but in this case $x_2$ is the other red vertex. Then $N(x_2,x_5)=\{x_1,x_3,x_4,x_6,x_9,x_{10},x_{14}\}$. There must be two red vertices in $ \mathcal{C}_5^4\setminus Y$ and these must be $\{x_{15},x_{17}\}$. Since $$N(x_2,x_5,x_{15},x_{17})=\{x_1,x_3,x_4,x_6,x_8,x_9,x_{10},x_{13},x_{14},x_{16},x_{18}\},$$ then the vertex $x_7$  in $\mathcal{C}_5^2\setminus Y$ must be red. But then we cannot have a red $x_{11}$. However, if  $x_{12}$ is red, then we have a problem because $x_{12}x_{19}\in E(G)$ and no vertex of $X$ can be red.  Consequently, this case does not have a solution, that is,  a chromatic class of size $7$ with $\{x_2,x_5\}$ colored red does not exist. 
 
 A similar analysis can be done if we assume $x_{13}$ instead of $x_{5}$ is red and another vertex in $\mathcal{C}_5^4$, and two different vertices in $ \mathcal{C}_5^1\setminus Y$.

 To finalize this case notice that between any two red vertices on the hamiltonian cycle $C: x_1, x_2, ... , x_{19}, x_1$, we have vertices of different chromatic classes and the number of vertices in each gap between any pair of red vertices are $\{1,2,3,1,1,1,3\}$ if we start after $x_3$ in the counter-clock direction.  So if we want a coloring with two chromatic classes of size $6$, it is necessary that if the colors are balanced in the spaces with one vertex (2 green and 2 blue), one of the spaces with three vertices, has two blue vertices and one green vertex and the other has two green vertices and one blue.  If $x_4$ is blue, then $x_5$ is green.  If $x_{13}\in \mathcal{B}$ then  $x_{17}\in \mathcal{G}$. Moreover if $x_7\in \mathcal{B}$ then $x_6\in \mathcal{G}$ and $x_{11}\in \mathcal{G}$, and also $x_9\in \mathcal{G}$ and $x_{10}\in \mathcal{B}$, but $x_2$ should be also green and also $x_{19}$ and $x_1$ is blue (see Figure \ref{robertsonredgreenblue} left side)
Then the green chromatic class has size  $7$ and the blue chromatic class has size $5$, and there does not exist a balanced coloring of the Robertson Graph with these characteristics in the red chromatic class.

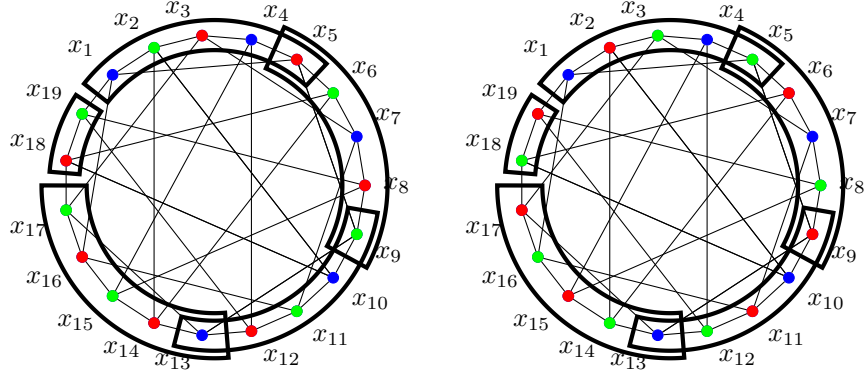
\begin{figure}
\centering
\begin{subfigure}{.49\textwidth}
\centering
\begin{tikzpicture}[scale=1]  
\usetikzlibrary{calc}
\def\r{2}

\foreach \i in {1,...,19}
	{\node[vert,blue,label=above:{}] (\i) at ({\r*cos(deg(2*pi*\i /19))},{\r*sin(deg(2*pi*\i/19))}) {};}	

\foreach \i in {1,...,19}
	{\node[label=deg(-2*pi*\i/19+2*pi*8 /19):{$x_{\i}$}] (x\i) at ({\r*cos(deg(-2*pi*\i /19+2*pi*8 /19))},{\r*sin(deg(-2*pi*\i/19+2*pi*8 /19))}) {};}

\foreach \i in {3,5,9,11,13,15,19}
	{\node[vert,red] (y\i) at ({\r*cos(deg(2*pi*\i /19))},{\r*sin(deg(2*pi*\i/19))}) {};}
	
\foreach \i in {2,6,8,10,12,16,18}
	{\node[vert,green] (y\i) at ({\r*cos(deg(2*pi*\i /19))},{\r*sin(deg(2*pi*\i/19))}) {};}

\foreach[remember=\x as \y (initially 1)] \x  in {2,...,19}
   	{
    \draw (\y) -- (\x);
    }
\draw (1) -- (19);

\foreach \x / \y in {1/5,1/16,2/13,2/9,3/18,3/7,4/15,4/12,6/17,6/13,7/11,8/19,8/15,9/17,10/5,10/14,14/18,16/11,17/9,17/6, 18/14,18/3,19/12}
    {
    \draw (\x) -- (\y);
    }

\begin{scope}[rotate=0]
\draw[ultra thick, domain=-3*pi:3*pi,samples=500] plot[domain=-3*pi/19:pi/19+3*2*pi/19] ({deg(\x)}:{2.3}) -- plot[domain=pi/19+3*2*pi/19:-3*pi/19] ({deg(\x)}:{1.7}) -- cycle;
\end{scope}

\begin{scope}[rotate=deg(4*2*pi/19)]
\draw[ultra thick, domain=-3*pi:3*pi,samples=500] plot[domain=-3*pi/19:pi/19+3*2*pi/19] ({deg(\x)}:{2.2}) -- plot[domain=pi/19+3*2*pi/19:-3*pi/19] ({deg(\x)}:{1.8}) -- cycle;
\end{scope}

\begin{scope}[rotate=deg(-4*2*pi/19)]
\draw[ultra thick, domain=-3*pi:3*pi,samples=500] plot[domain=-3*pi/19:pi/19+3*2*pi/19] ({deg(\x)}:{2.2}) -- plot[domain=pi/19+3*2*pi/19:-3*pi/19] ({deg(\x)}:{1.8}) -- cycle;
\end{scope}

\begin{scope}[rotate=deg(-8*2*pi/19)]
\draw[ultra thick, domain=-3*pi:3*pi,samples=500] plot[domain=-3*pi/19:pi/19+3*2*pi/19] ({deg(\x)}:{2.3}) -- plot[domain=pi/19+3*2*pi/19:-3*pi/19] ({deg(\x)}:{1.7}) -- cycle;
\end{scope}

\begin{scope}[rotate=deg(-11.25*2*pi/19)]
\draw[ultra thick, domain=-3*pi:3*pi,samples=500] plot[domain=0:3*pi/19] ({deg(\x)}:{2.2}) -- plot[domain=3*pi/19:0] ({deg(\x)}:{1.8}) -- cycle;
\end{scope}

\end{tikzpicture}
\end{subfigure}
\begin{subfigure}{.49\textwidth}
\centering
\begin{tikzpicture}[scale=1]  
\usetikzlibrary{calc}
\def\r{2}

\foreach \i in {1,...,19}
	{\node[vert,blue,label=above:{}] (\i) at ({\r*cos(deg(2*pi*\i /19))},{\r*sin(deg(2*pi*\i/19))}) {};}	

\foreach \i in {1,...,19}
	{\node[label=deg(-2*pi*\i/19+2*pi*8 /19):{$x_{\i}$}] (x\i) at ({\r*cos(deg(-2*pi*\i /19+2*pi*8 /19))},{\r*sin(deg(-2*pi*\i/19+2*pi*8 /19))}) {};}

\foreach \i in {2,6,8,10,12,16,18}
	{\node[vert,red] (y\i) at ({\r*cos(deg(2*pi*\i /19))},{\r*sin(deg(2*pi*\i/19))}) {};}
	
\foreach \i in {3,5,9,11,13,15,19}
	{\node[vert,green] (y\i) at ({\r*cos(deg(2*pi*\i /19))},{\r*sin(deg(2*pi*\i/19))}) {};}

\foreach[remember=\x as \y (initially 1)] \x  in {2,...,19}
   	{
    \draw (\y) -- (\x);
    }
\draw (1) -- (19);

\foreach \x / \y in {1/5,1/16,2/13,2/9,3/18,3/7,4/15,4/12,6/17,6/13,7/11,8/19,8/15,9/17,10/5,10/14,14/18,16/11,17/9,17/6, 18/14,18/3,19/12}
    {
    \draw (\x) -- (\y);
    }

%\draw[rounded corners={10pt}, very thick] (.68,1.56) foreach \x/\y in {1.56/.61,1.58/-.74,2.26/-.92,2.38/.8,1.09/1.99} {-- (\x,\y)} 
%            -- cycle;
%
%\draw[rounded corners={10pt}, very thick] (-.15,-2.62) foreach \x/\y in {1.37/-2.12, 2.4/-.88, 1.75/-.24, .9/-1.27, -.42/-1.7} {-- (\x,\y)} 
%            -- cycle;
%    
%\draw[rounded corners={10pt}, very thick] (-1.7,0) foreach \x/\y in {-1.38/1.05, -1.88/1.38,-2.42/.19} {-- (\x,\y)} 
%            -- cycle;    
%            
%\draw[rounded corners={10pt}, very thick] (-.46,-2.63) foreach \x/\y in {-2.38/-.51,-1.78/0.1,.41/-2.09} {-- (\x,\y)} 
%            -- cycle;  
%            
%\draw[rounded corners={10pt}, very thick] (-1.5,1.09) foreach \x/\y in {1.31/1.33, 1.3/2.19, -1.48/1.97} {-- (\x,\y)} 
%            -- cycle;                                          
%        

\begin{scope}[rotate=0]
\draw[ultra thick, domain=-3*pi:3*pi,samples=500] plot[domain=-3*pi/19:pi/19+3*2*pi/19] ({deg(\x)}:{2.3}) -- plot[domain=pi/19+3*2*pi/19:-3*pi/19] ({deg(\x)}:{1.7}) -- cycle;
\end{scope}

\begin{scope}[rotate=deg(4*2*pi/19)]
\draw[ultra thick, domain=-3*pi:3*pi,samples=500] plot[domain=-3*pi/19:pi/19+3*2*pi/19] ({deg(\x)}:{2.2}) -- plot[domain=pi/19+3*2*pi/19:-3*pi/19] ({deg(\x)}:{1.8}) -- cycle;
\end{scope}

\begin{scope}[rotate=deg(-4*2*pi/19)]
\draw[ultra thick, domain=-3*pi:3*pi,samples=500] plot[domain=-3*pi/19:pi/19+3*2*pi/19] ({deg(\x)}:{2.2}) -- plot[domain=pi/19+3*2*pi/19:-3*pi/19] ({deg(\x)}:{1.8}) -- cycle;
\end{scope}

\begin{scope}[rotate=deg(-8*2*pi/19)]
\draw[ultra thick, domain=-3*pi:3*pi,samples=500] plot[domain=-3*pi/19:pi/19+3*2*pi/19] ({deg(\x)}:{2.3}) -- plot[domain=pi/19+3*2*pi/19:-3*pi/19] ({deg(\x)}:{1.7}) -- cycle;
\end{scope}

\begin{scope}[rotate=deg(-11.25*2*pi/19)]
\draw[ultra thick, domain=-3*pi:3*pi,samples=500] plot[domain=0:3*pi/19] ({deg(\x)}:{2.2}) -- plot[domain=3*pi/19:0] ({deg(\x)}:{1.8}) -- cycle;
\end{scope}      
\end{tikzpicture}
\end{subfigure}
\caption {\label{robertsonredgreenblue} Two different cases of a 3-non equitable coloring of Robertson cage}
\end{figure}

\item Suppose that $x_9$ is red, and there is  one other red vertex in $\mathcal{C}_5^2$. In this first case we assume $x_7$ is also red so that the unique possible red vertex in $\mathcal{C}_5^3$ is $x_{12}$, and we have the following $N(x_7,x_9,x_{12})=\{x_3,x_4,x_5,x_6,x_8,x_{10},x_{11},x_{13},x_{19}\}$. As an immediate conclusion,
$x_{18}$ must be a red vertex in $X$
and we must have
two red vertices in $\mathcal{C}_5^4\setminus Y$, and the options are only $x_{14}$ and $x_{16}$, because $x_{17}\in N(x_{18})$. Then,at this point, we have: 
$$N(x_7,x_9,x_{12},x_{14},x_{16},x_{18})=\{x_1,x_2,x_3,x_4,x_5,x_6,x_8,x_{10},x_{11},x_{13},x_{15},x_{17},x_{19}\}$$ and there can be no red vertex in $\mathcal{C}_5^1\setminus Y$ and this case does not have solution.

The second case is, assuming $x_9$ is red, that $x_6$ is the other red vertex in $\mathcal{C}_5^2$.  As, $x_6x_{18}\in E(G)$ then $x_{19}\in X\cap \mathcal{R}$ and 
$N(x_6,x_9,x_{19})=\{x_1,x_5,x_7,x_8,x_{10},x_{12},x_{13},x_{14},x_{18}\}$.  Then, the other red vertex in $\mathcal{C}_5^3$ only can be $x_{11}$, and the two red vertices in $ \mathcal{C}_5^4\setminus Y$ are $x_{15}$ and $x_{17}$. Then: 
$$N(x_6,x_9x_{11},x_{15},x_{17},x_{19})=\{x_1,x_3,x_4,x_5,x_7,x_8,x_{10},x_{12},x_{13},x_{14},x_{16},x_{18}\}.$$
Consequently, the red vertex in $\mathcal{C}_5^1$ must be $x_2$ and everything is determined (see Figure \ref{robertsonredgreenblue} right side). A similar analysis can be done if we begin with  $x_9$ as red,  consider firstly the second red vertex in $\mathcal{C}_5^3$ instead of starting with $\mathcal{C}_5^3$, then we have two options and only one will be possible.

To finalize we will prove, as in the previous case, that it will be impossible to find an equitable coloring with this distribution of the red chromatic class in $G$. Here the number of vertices in each gap between any pair of red vertices are $\{1,3,2,1,3,1,1\}$ starting in $x_1$ in the counter-clock direction (the same as in the previous case) and also here if we want a coloring with two chromatic classes of length $6$ it is necessary that if we balanced the colors in the spaces with one vertex (2 green and 2 blue), one of the spaces with three vertices, has two blue vertices and one green vertex and the other should have two green vertices and one blue.  Here, if $x_1\in \mathcal{B}$, then $\{x_5,x_7\}\in \mathcal{G}$ and $x_4\in \mathcal{B}$  and $x_3\in \mathcal{G}$. Immediately $\{x_{12},x_{14}\}\in \mathcal{G}$ and  $x_{15}\in \mathcal{B}$. But, as $x_{10}x_{18}\in E(G)$ and both are vertices between red vertices, one of them is blue and the other green (w.l.o.g. $x_{10}\in \mathcal{B}$ and $x_{18}\in \mathcal{G}$). Finally,  $x_7\in \mathcal{B}$ and $x_8\in \mathcal{G}$ (see Figure \ref{robertsonredgreenblue}, right side).

\medskip

Then, newly the green chromatic class has size also $7$ and the blue chromatic class has size $5$, and neither in this case, does not there exist a balanced coloring of Robertson Graph.

\end{itemize}
\end{itemize}

\end{proof}

On Figure \ref{fig:453} is a $(4,5,3)$-equitable graph of order $20$ and three chromatic classes of size $6$, $7$ and $7$. As the $(4,5,3)$-cage graph (the Robertson graph) is not equitable, then $G$ is a $(4,5,3)$-equitable cage. 

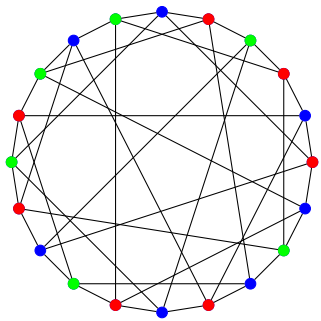
\begin{figure}
\centering
\begin{tikzpicture}[scale=2]  
\foreach \i in {1,...,20}
	{\node[vert, blue] (\i) at ({1*cos(deg(2*pi*\i /20))},{1*sin(deg(2*pi*\i/20))}) {};}
	
\foreach \i in {16, 2, 9, 20, 4, 14, 11}
	{\node[vert, red] (\i) at ({1*cos(deg(2*pi*\i /20))},{1*sin(deg(2*pi*\i/20))}) {};}

\foreach \i in {3, 10, 13, 18, 6, 8}
	{\node[vert, green] (\i) at ({1*cos(deg(2*pi*\i /20))},{1*sin(deg(2*pi*\i/20))}) {};}	

\foreach \x / \y in {16/1,1/20,16/15,16/7,16/17,1/9,1/2,20/5,20/12,20/19,15/3,15/10,15/14,7/8,7/11,7/6,17/4,17/18,17/13,9/8,9/10,9/13,2/3,2/18,2/6,5/4,5/10,5/6,12/3,12/11,12/13,19/8,19/18,19/14,8/4,4/3,10/11,11/18,13/14,14/6}
    {
    \draw (\x) -- (\y);
    }

\end{tikzpicture}
\caption{\label{fig:453} The $(4,5,3)$-equitable cage}
\end{figure}

%\Section{The case $r=4$ and $g\geq 5$}
\section{Future work.}

Obviously,  natural future research is working with regular graphs of odd girth and a chromatic number greater or equal to $4$, because in this paper we only worked with graphs of chromatic number equal to $3$.

For example, for girth $5$ the first cage that has chromatic number equal to $4$ is the $(5,5)$-cage which has order $30$.   Obviously, it is a $(5,5,4)$-cage, but we do not know if it is equitable or not. 
Moreover, there exist a lot of nice questions about the chromatic number of the cages, that is, about the chromatic number of the minimum $(k;g)$-graphs known until this moment (for any $k\geq 3$ and odd $g$).

On the other hand, when $g$ is even and $\chi\not= 2$, the problem of finding  $(r,g,\chi)$-graphs, or $(r,g,\chi)$-cages, is naturally related to the problem of finding cages with given girth pair. These graphs were introduced in \cite{HK83} and studied, for example in \cite{BS14,C97}. The $(r;g,h)$-\textbf{cages} are $r$-regular graphs of minimum order and fixed girth $g$, where $h$ is the length of the smallest cycle of different parity of $g$ and $g<h$. Obviously, if $g$ is even, an $(r;g,h)$-graph $G$ (in this context) could be an $(r,g,3)$-graph (in our context), if the chromatic number of $G$ is $3$.  For us, a natural future research question is to study the relationship between these two ``kinds" of cages.

\bibliographystyle{alpha}

\end{document}